\documentclass[11pt]{article}
\usepackage{fullpage}
\usepackage{graphicx,epsfig,psfrag,rotate,xcolor}

\usepackage{amsmath,amsfonts,amssymb,latexsym,amsthm}
\usepackage{setspace,enumerate,ifthen,subfig}
\usepackage{hyperref}
\usepackage{cite}
\usepackage{ifpdf}
\usepackage{subfiles}
\usepackage{enumitem}

\usepackage{algorithm}
\usepackage[noend]{algorithmic}
\usepackage{csquotes}
\usepackage{soul} 

\usepackage{tikz}
\usetikzlibrary{shapes,arrows}
\usetikzlibrary{shapes.geometric}
\tikzstyle{line} = [draw, -latex']
\usepackage{tikz}
\usetikzlibrary{shapes,arrows}
\tikzstyle{decision} = [diamond, draw, fill=white, 
    text width=4em, text badly centered, node distance=3cm, inner sep=0pt]
\tikzstyle{block} = [rectangle, draw, fill=blue!20, 
    text width=8em, text centered, rounded corners, minimum height=4em]
\tikzstyle{blank} = [rectangle, fill=white, 
    text width=8em, text centered, rounded corners, minimum height=4em]
\tikzstyle{line} = [draw, -latex']
\tikzstyle{vague} = [draw, -latex']
\tikzstyle{cloud} = [draw, ellipse,fill=white, node distance=3cm,
    minimum height=2em]

\definecolor{Blu}{rgb}{.255,.41,.884} 

\newtheorem{theorem}{Theorem}[section]
\newtheorem{assumption}{Assumption}[section]
\newtheorem{proposition}{Proposition}[section]

\newcommand{\R}{\mathbb{R}}

\newcommand{\stt}{\text{ s.t. }}

\newcommand{\from}{\mathrel{\mathop:}=}

\newcommand{\sumgt}{\displaystyle\sum_{g \in G}\displaystyle\sum_{t \in T}}
\newcommand{\sumkt}{\displaystyle\sum_{k \in G}\displaystyle\sum_{t \in T}}
\newcommand{\sumt}{\displaystyle\sum_{t \in T}}
\newcommand{\sumg}{\displaystyle\sum_{g \in G}}
\newcommand{\sumk}{\displaystyle\sum_{k \in G}}

\allowdisplaybreaks

\begin{document}

\title{Transmission-Constrained Unit Commitment}

\author{Claudio Gambella, Jakub Marecek, Martin Mevissen, \\ Jose Maria Fernandez Ortega, Sara Pezic Djukic, Mustafa Pezic}
\maketitle
\thispagestyle{empty}
\pagestyle{empty}

\begin{abstract}
The unit commitment with transmission constraints in the alternating-current (AC) model
is a challenging mixed-integer non-linear optimisation problem. 
We present an approach based on decomposition of a 
Mixed-Integer Semidefinite Programming (MISDP) problem into a mixed-integer quadratic (MIQP) master problem 
and a semidefinite programming (SDP) sub-problem.
Between the master problem and the sub-problem, we pass novel classes of cuts.
We analyse finite convergence to the optimum of the MISDP and 
report promising computational results on a test case from the Canary Islands, Spain.
\end{abstract}

\section{Introduction}

In both traditional power systems and modern microgrids, a single operator decides what generating units to turn on and off and when. 
This decision, known as the unit commitment \cite{8275570}, is presently either divorced from the feasibility of the power flows,
 or considers only a crude linear approximation of the non-linear power flows. 
Once the unit commitment is obtained, its feasibility with respect to 
 the transmission constraints has to be tested. 
Should no feasible solutions be available, manual modifications of the optimal schedule are needed.

In relation to the introduction of intermittent renewable energy sources and the consequent 
rapid changes in the (directions of) power flows, 
there has been much interest in making decisions concerning generation while considering transmission constraints explicitly.
The prototypical problem in this field, which is known variously as the Transmission-Constrained Unit Commitment (TCUC) 
in the Alternating-Current (AC) model, 
the Network-Constrained Unit Commitment (NCUC), 
or the Unit Commitment with Optimal Power Flows (UC+OPF),
can be cast as a large Mixed-Integer Non-Linear Programming (MINLP) problem.

The state-of-the-art approaches to this MINLP are based on generalized Benders' decompositions \cite{Geo72}.
There, the MINLP is decomposed into the mixed-integer linear programming (MILP) part 
and the continuous non-linear programming (NLP) part, but one cannot guarantee convergence to the global optimum.
The MILP part is known as the master problem and the NLP part is known as the sub-problem.
Several papers consider active power flow and reactive power flow in isolation 
\cite{761909,1490608,1425598,4349079,5372005,7061968}, which makes it possible to consider linear programming
(LP) in the sub-problem and LP duality in deriving no-good cuts.
In hydro-power scheduling \cite{6919349,6709763}, Semidefinite Programming (SDP) has been used in the sub-problem.
More recently, Castillo et al. \cite{7377125} considered Benders' decomposition with non-linear sub-problems,
solved either by piece-wise linearisation or using NLP solvers with local convergence.
Despite this long history of research, this approach has not allowed for convergence guarantees, yet. 

We present a principled approach, which decomposes a mixed-integer semidefinite-programming problem (referred to as TCUC-SDP in Figure \ref{fig-TCUC-models})
into a mixed-integer quadratic programming (MIQP) master problem 
and SDP sub-problems, wherein values of discrete variables are fixed.
Whenever a candidate solution is found in the master problem, 
the SDP sub-problem is solved.
When the sub-problem turns out to be feasible, it completes the candidate solution and provides an upper bound 
on the optimum of the TCUC-SDP 
and the cost of non-revenue power for each period of the TCUC-SDP.
The cost can then be added to the objective of the master problem, whenever the same per-period unit commitment 
is considered again.
When the sub-problem is infeasible, a corresponding part of the feasible region of the master problem should be removed.
To pass these details from the SDP sub-problem back to the master problem, we derive several novel classes of cuts.


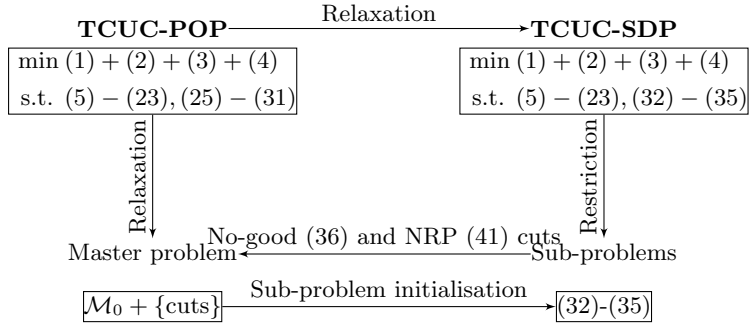
\begin{figure}[htb]
	\begin{footnotesize}
		\begin{tikzpicture}[inner sep=0pt,minimum size=4mm]]
		
		\node[draw,shape=rectangle] (TCUC-POP) at (0,3) {
			$\begin{aligned}
			\min \: & \eqref{CoperP_noCommitmentCost}+\eqref{CoperP_CommitmentCost}+\eqref{Carr0}+\eqref{Carr} \\
			\stt& \eqref{LoadBalance}-\eqref{varR}, \eqref{defPgTrans}-\eqref{defPgTransLast}
			\end{aligned}$
		};
		\node[draw=none, fill=none] (TCUC-POP-label) at (0,3.7) {\textbf{TCUC-POP}};

		\node[draw,shape=rectangle] (TCUC-SDP) at (6,3) {
	$\begin{aligned}
	\min \: & \eqref{CoperP_noCommitmentCost}+\eqref{CoperP_CommitmentCost}+\eqref{Carr0}+\eqref{Carr} \\
	\stt& \eqref{LoadBalance}-\eqref{varR}, \eqref{defA}-\eqref{varBk}
	\end{aligned}$
};
\node[draw=none, fill=none] (TCUC-SDP-label) at (6,3.7) {\textbf{TCUC-SDP}};

		\path [line] (TCUC-POP-label) -- (TCUC-SDP-label) node[pos=0.5,above] {Relaxation}; 

		\node[draw,shape=rectangle] (master) at (0,0) {
		$\mathcal{M}_0 + \{\text{cuts}\}$
		};
		\node[draw=none, fill=none] (master-label) at (0,0.7) {Master problem};
		
		\path [line] (TCUC-POP) -- (master-label) node[pos=0.5,above,rotate=+90] {Relaxation};
		
		\node[draw,shape=rectangle] (sub-problem) at (6,0) {
		\eqref{defA}-\eqref{varBk}
		};
		\node[draw=none, fill=none] (sub-problem-label) at (6,0.7) {Sub-problems}; 
		
		\path [line] (TCUC-SDP) -- (sub-problem-label) node[pos=0.5,above,rotate=+90] {Restriction};
		
		\path [line] (sub-problem-label) -- (master-label) node[pos=0.5,above] {No-good \eqref{nogood} and NRP \eqref{NRP-cut} cuts};
		\path [line] (master) -- (sub-problem) node[pos=0.5,above] {Sub-problem initialisation};

		\end{tikzpicture}
	\end{footnotesize}
	\caption{Our approach illustrated.}\label{fig-TCUC-models}
\end{figure}


After we introduce the requisite notation, we formalise the approach in Section
\ref{sec:Decomposition} and prove its convergence in Section \ref{sec:Convergence}.
Subsequently, we present the results obtained with the approach on test instances from the Canary Islands.

\section{Problem Formulation}
\label{sec:prob}

We consider the representation of a power system used by \cite{lavaei2012zero,Molzahn2011,GhaMM16} and the corresponding notation.  
The power system network is represented by a directed graph, where each vertex $k \in N$ is called a ``bus'' and each directed edge 
$(l, m) \in E \subseteq N \times N$ is called a ``branch'', and each branch
can have an ideal phase-shifting transformer at its ``from'' end and is modeled as a $\Pi$-equivalent circuit. 
Let $G \subseteq N$ be the set of generators, also known as PV buses. Let $L \subseteq E$ be the set of branches on which a limitation on the generation of the thermal power is present.
The remainder of vertices $N \setminus G$ represent the PD buses.
We discretise time to a set of time periods $T$, starting from $T_i$ and ending in $T_f$.

Using the usual rectangular power-voltage formulation of power flows in each period, the decision variables are:\\
\begin{tabular}{cp{0.5\linewidth}}
  $u_k^t$ &	 commitment status of generator $k \in G$ at time $t \in T$,\\
  $R_{+}^{t}$, $R_{-}^{t}$ & spinning up/down-reserve at time $t$,\\
    $x^t=\{\Re{V_{k}^{t}}+ j \Im{V_{k}^{t}}\}_{k \in N}$ & vector of voltages $V_{k}^{t}$ at time $t \in T$,\\
  $ (P^{g,t}_k, Q^{g,t}_k)$ & active and reactive power of the generator at bus $k \in N$ at time $t \in T$,\\
  $ (P_{lm}^{t}, Q_{lm}^{t})$ & active and reactive power flow on $(l,m) \in E$ at time $t \in T$,
\end{tabular}

wherein  $S_{k}^{g,t}$, $S_{lm}^{t}$ are not used explicitly, but rather expressed in terms of their components and $x^t$ in all actual computations. 
For notational convenience, we use $u=\{u_k^t: k \in G, t \in T\}$ and $P = \{P_k^{g,t}: k \in G, t \in T\}$, respectively, for the status and active power of all generators during the time periods in $T$.

To complete the formulation, we need to introduce a number of power-flow-related constants:\\
\begin{tabular}{cp{0.6\linewidth}}
$y \in \mathbb{R}^{|N|\times|N|}$ & network admittance matrix\\
 $\bar{b}_{lm}$ & shunt element value at branch $(l,m) \in E $\\
 $g_{lm}+jb_{lm}$ & the series admittance on a branch $(l,m) \in E$\\
 $P^{d,t}_k$ & active load (demand) at bus $k \in N$ on period $t$\\
  $Q^{d,t}_k$ & reactive load (demand) at bus $k \in N$ on period $t$\\
   $P^{d,t}$ &  aggregate active demand on period $t$\\
  $c_k^2, c_k^1, c_k^0, $ & coefficients of the quadratic generation costs $C_k$ at generator $k$\\
   $P_k^{\min, t}$, $P_k^{\max, t}$ & limits on active generation at bus $k$ on period $t$\\
  $Q_k^{\min, t}$, $Q_k^{\max, t}$ & limits on reactive generation at bus $k$ on period $t$\\
  $V_k^{\min}$, $V_k^{\max}$ & limits on the absolute value of the voltage at bus $k  \in N$\\
  $S_{lm}^{\max}$ & limit on the absolute value of the apparent power of branch $(l,m) \in L$,
\end{tabular}\\\vspace{0.2cm}

and a number of unit-commitment-related constants:\\\vspace{0.2cm}
\begin{tabular}{cp{0.6\linewidth}}
$u_k^{\textrm{Carr}}$ & start-up cost of generator $k$\\ 
$u_k^{\textrm{Init}}$ & initial status of generator $k$\\ 
$u_k^{\textrm{InitP}}$ & initial power of generator $k$\\ 
$u_k^{\textrm{Inertia}}$ & flag indicating whether generator $k$ provides inertia\\ 
$u_k^{\textrm{RampUp}}$ & ramp-up value of generator $k$\\
$u_k^{\textrm{minOFF}}$ & minimum number of periods of inactivity for generator $k$ \\
$u_k^{\textrm{minON}}$ & minimum number of periods of activity for generator $k$ \\
$u_k^{\textrm{InitT}}$ & initial period of generator $k$ \\ 
$\overline{R_k^{t}}$ & maximum limit for reserve at generator $k$ on period $t$\\ 
$\underline{R_k^{t}}$ & minimum limit for reserve at generator $k$ on period \\ $t$ 
$\mathcal{R}_{+}^{t}$  & up secondary reserve on period $t$ 
\end{tabular}\\
\begin{tabular}{cp{0.6\linewidth}}
$\mathcal{R}_{-}^{t}$  & down secondary reserve on period $t$\\
$\underline{U^+}$ & minimum number of generators on\\
$P_{\%}$ & maximum power in demand percentage\\
$C_{\textrm{OperP}}$ & operational costs of generators units\\
$C_{\textrm{Arr}}$ & start-up costs of generator units,
\end{tabular}

and the following sets of indices:
\begin{align*}
\begin{split}
	\textrm{Ramp}_{T_i}=\{&k \in G: u_k^{\textrm{RampUp}}<(P_k^{\max}-P_k^{\min}),\\
	& u_k^{\textrm{Init}}=1\},\\
	\textrm{Ramp}=\{&t \in T \setminus T_i, k \in G:\\&u_k^{\textrm{RampUp}}<(P_k^{\max}-P_k^{\min})\},\\
	\textrm{MinOff}_{\textrm{Init}0}=\bigl\{& k \in G: (u_k^{\textrm{minOFF}}>1,\\& u_k^{\textrm{minOFF}}-u_k^{\textrm{InitT}}>0, u_k^{\textrm{Init}}=0), \\
	&s \in \{T_i,\dots, T_i+u_k^{\textrm{minOFF}}-u_k^{\textrm{InitT}}-1\}\bigr\}\\
	\textrm{MinOff}_{\textrm{Init}1}=\bigl\{& k \in G: (u_k^{\textrm{minOFF}}>1, u_k^{\textrm{Init}}=1),\nonumber\\ 
	& s \in \{T_i+1,\dots ,T_i+u_k^{\textrm{minOFF}}-1 \}\bigr\}   \\
	\textrm{MinOff}=\bigl\{& k \in G: u_k^{\textrm{minOFF}}>1, t \in T\setminus T_i\nonumber\\
	& s \in \{t+1,\dots ,t+u_k^{\textrm{minOFF}}-1\}\bigr\} \\
	\textrm{MinOn}_{\textrm{Init}1}=&\bigl\{k \in G: (u_k^{\textrm{minON}}>1,\\&u_k^{\textrm{minON}}-u_k^{\textrm{Init}T}>0, u_k^{\textrm{Init}}=1),\\
	&s \in \{T_i, \dots,  T_i+u_k^{\textrm{minOFF}}-u_k^{\textrm{Init}T}-1\}\bigr\}\\
	\textrm{MinOn}_{\textrm{Init}0}=&\bigl\{k \in G: (u_k^{\textrm{minON}}>1, u_k^{\textrm{Init}}=0),\\
	& s \in \{T_i+1,\dots ,T_i+u_k^{\textrm{minON}}-1\}\bigr\}\\
	\textrm{MinOn}=&\bigl\{k \in G: u_k^{\textrm{minON}}>1, t \in T\setminus T_i,\\
	& s \in \{t+1,\dots ,t+u_k^{\textrm{minON}}-1\}\bigr\}
\end{split}
\end{align*}

The complete formulation as a Polynomial Optimisation Problem (POP) is as follows:

\begin{align}
\text{TCUC-POP}: \min \: & \sumkt (c_k^2 (P^{g,t}_k)^2 + c_k^1 P^{g,t}_k) & \label{CoperP_noCommitmentCost}  \\
&+ \sumkt  c_k^0 \ u_k^t & \label{CoperP_CommitmentCost}\\
&+\sumgt u_k^{\textrm{Carr}} \cdot \max(u_k^t-u_k^{t-1}, 0) & \label{Carr0}\\
&+\displaystyle\sum_{k \in G}  u_k^{\textrm{Carr}} \cdot \max(u_k^1-u_k^{\textrm{Init}}, 0) & \label{Carr} 
\end{align}
\begin{align}
\stt & \sumk P_k^{g,t} = P_{}^{g, t} &t \in T\label{LoadBalance}\\ 
& R_{+}^{t} = \sumk (\overline{R_k^{t}} u_k^t -P_k^{g,t}) &t \in T\label{UpReserve} \\ 
& R_{+}^{t} \geq \mathcal{R}_{+}^{t} & t \in T\label{MinUpReserve}\\
& R_{-}^{t} = \sumk (P_k^{g,t}-\underline{R_k^{t}} u_k^t)  &t \in T\label{DownReserve}\\
& R_{-}^{t} \geq \mathcal{R}_{-}^{t}  &t \in T\label{MinDownReserve}\\
& P_k^{g,t} \geq P_k^{\min, t} u_g^t &t \in T, k \in G\label{Pmin}\\
&P_k^{g,t} \leq P_k^{\max} u_g^t &t \in T, k \in G\label{Pmax}\\
& P_k^{g,t}\leq P_{\%} \cdot P ^{d,t} u_g^t  &t \in T, k \in G\label{Ppercentage}\\
&P_k^{T_i}-u_k^{\textrm{InitP}}\leq u_k^{\textrm{RampUp}} &k \in \textrm{Ramp}_{T_i} \label{Rampa_Sub_1}\\
&P_k^{t}-P_k^{t-1}\leq u_k^{\textrm{RampUp}}&k \in \textrm{Ramp}_{T_i}\label{Rampa_Sub}\\
&\sumg u_g^{\textrm{Inertia}} u_g^t  \geq \underline{U^+} &t \in T\label{MinUON}\\
&u_k^s=0 &(k,s) \in \textrm{MinOff}_{\textrm{Init}0}\label{MinOff1}\\
&u_k^{\textrm{Init}}-u_k^{T_i}\leq1-u_k^s &(k,s) \in \textrm{MinOff}_{\textrm{Init}1}\label{MinOff2}\\
& u_k^{t-1}-u_k^{t}\leq1-u_k^s & (k,t,s) \in \textrm{MinOff}\label{MinOff}\\
&u_k^s=1 &(k,s)\in \textrm{MinOn}_{\textrm{Init}1} \label{MinOn1}\\
&u_k^{T_i}-u_k^{\textrm{Init}}\leq u_k^s &(k,s)\in \textrm{MinOn}_{\textrm{Init}0} \label{MinOn2}\\
& u_k^{t}-u_k^{t-1}\leq u_k^s &(k,t,s) \in \textrm{MinOn}\label{MinOn}\\
& P_k^{g,t}\geq 0, u_k^{t}\in \{0,1\}&k \in G, t \in T\label{varPu}\\
& R_+^{t}\geq 0,  R_-^{t}\geq 0  &t \in T\label{varR}& \\
& \textrm{transmission constraints}& t \in T. \label{transcons}
\end{align}

The objective function is composed of operational costs due to the quadratic power generations costs \eqref{CoperP_noCommitmentCost} and commitment costs \eqref{CoperP_CommitmentCost}, and of start-up costs \eqref{Carr0}, \eqref{Carr}; in the actual implementation, the maximization 
operators in the start-up costs are equivalently expressed via artificial variables subject to linear constraints.
Constraints \eqref{LoadBalance} ensure that the active power generated matches the demand. 
Constraints \eqref{UpReserve}, \eqref{MinUpReserve} express the limits on spinning up-reserve created by generators, 
while \eqref{DownReserve} and \eqref{MinDownReserve} account for the down-reserve. 
Constraints \eqref{Pmin}-\eqref{Ppercentage} impose the bounds on the active power generated, where the upper bound may be required to be refined to a percentage of on the aggregate load. 
The so-called \enquote{ramping} constraints \eqref{Rampa_Sub_1} and \eqref{Rampa_Sub} ensure that, for each generator, 
the increase of active power in consecutive periods does not exceed the threshold $u_k^{\textrm{RampUp}}$. 
Constraints \ref{MinUON} impose that at least $\underline{U^+}$ generators are on in each time period. 
Constraints \eqref{MinOff1}, \eqref{MinOff2}, \eqref{MinOff} enforce requirements on the minimum number of hours of inactivity of generators,
while \eqref{MinOn1}, \eqref{MinOn2}, \eqref{MinOn} take the minimum number of periods of activity into account. 
Constraints \eqref{varPu}, \eqref{varR} bound the domain of the decision variables.

The transmission constraints \eqref{transcons} can be stated as:
\begin{align}
& P_k^{g, t} = \text{tr}(Y_kx^t(x^t)^\top)+P_k^{d, t} \label{defPgTrans}\\
&P_k^{\min, t} \leq \text{tr}(Y_kx^t(x^t)^\top)+P_k^{d, t} \leq P_k^{\max, t} \label{boundPg}\\
& Q_k^{\min, t}\leq \text{tr}(\bar{Y}_kx^t(x^t)^\top)+Q_k^{d, t}\leq Q_k^{\max, t} \label{boundQg}\\
&(V_k^{\min})^2 \leq \text{tr}(M_kx^t(x^t)^\top) \leq  (V_k^{\max})^2 \label{boundV}\\
&(P_{lm}^t)^2+(Q_{lm}^t)^2 \leq (S_{lm}^{\max})^2 \label{boundSlm}\\
& P_{lm}^t = \text{tr}(Y_{lm}x^t(x^t)^\top) \label{Plm}\\
& Q_{lm}^t = \text{tr}(\bar{Y}_{lm}x^t(x^t)^\top), \label{defPgTransLast}
\end{align}
where $^\top$ denotes the transpose of a vector.
This formulation can be seen as an extension of $OP_2$ in \cite{GhaMM16},
which can be relaxed to [OP$_2$-H$_1$]$^*$ of \cite{GhaMM16}, which
is equivalent with Optimization 4 in \cite{lavaei2012zero}.
We refer to \cite{6815671} for a discussion.

In our extension of [OP$_2$-H$_1$]$^*$, the SDP involves matricial decision variables:
\begin{align}
A^t \in \mathbb{R}^{2|N| \times 2|N|} & \quad t \in T \notag\\
B^{t}(b) \in \mathbb{R}^{3 \times 3} & \quad b \in \{1, \dots, 2|L|\}, t \in T \notag\\
B_k^{t}(k) \in \mathbb{R}^{2 \times 2} & \quad k \in G, t \in T, \notag
\end{align}
which are constrained to be positive semidefinite and wherein elements of the latter are labelled as follows:
\begin{align}
B^{t}(b) & = 
\begin{bmatrix}
B_{lm}^{t, 1}(b) & B_{lm}^{t, 2}(b) & B_{lm}^{t, 3}(b)\\
B_{lm}^{t, 2}(b) & B_{lm}^{t, 4}(b) & B_{lm}^{t, 5}(b)\\
B_{lm}^{t, 3}(b) & B_{lm}^{t, 5}(b) & B_{lm}^{t, 6}(b)
\end{bmatrix} &  \notag\\
B_k^{t}(k) & = 
\begin{bmatrix}
1 & B_{k}^{t, 1}(k)\\
B_{k}^{t, 1}(k) & B_{k}^{t, 2}(k)
\end{bmatrix}, \notag
\end{align}
and non-negative multipliers $\overline{\lambda}_k^t, \underline{\lambda}_k^t, \overline{\gamma}_k^t, \underline{\gamma}_k^t, \overline{\mu}_k^t, \underline{\mu}_k^t$, $k \in G, t \in T$. 
The dimension of the SDP relaxation is hence $2|T|+2|L||T|+|G||T|$.
For period $t \in T$, the constraints are: 
\begin{align}
A^t = & \sum_{k \in N} \overline{\lambda}_k^t Y_k -\underline{\lambda}_k^t Y_k+
\sum_{k \in G} c_1^k Y_k + 2 Y_k R_k^{t, 1}(k) \sqrt{c^2_k}\nonumber\\ 
&+\sum_{k \in N} \overline{\gamma}_k^t \bar{Y}_k -\underline{\gamma}_k^t \bar{Y}_k 
+ \sum_{k \in N} \overline{\mu}_k^t M_k(k) -\underline{\mu}_k^t M_k(k)\nonumber\\
&+\sum_{b \in \{1, \dots, 2|L|\}} \left(2R_{lm}^{t, 2}(b) Y_{lm}  +2R_{lm}^{t, 3}(b) \bar{Y}_{lm}\right)\label{defA}
\end{align}
\begin{align}
A^t & \succeq 0 \label{sdp2} \\
\begin{bmatrix}
B_{lm}^{t, 1}(b) & B_{lm}^{t, 2}(b) & B_{lm}^{t, 3}(b)\\
B_{lm}^{t, 2}(b) & B_{lm}^{t, 4}(b) & B_{lm}^{t, 5}(b)\\
B_{lm}^{t, 3}(b) & B_{lm}^{t, 5}(b) & B_{lm}^{t, 6}(b)
\end{bmatrix}
& \succeq 0\\
\begin{bmatrix}
1 & B_{k}^{t, 1}(k)\\
B_{k}^{t, 1}(k) & B_{k}^{t, 2}(k)
\end{bmatrix}
& \succeq 0.\label{varBk}
\end{align}
We denote the Mixed-Integer Semidefinite Program (MISDP) composed of (\ref{CoperP_noCommitmentCost}--\ref{varR}) and (\ref{defA}--\ref{varBk}) as TCUC-SDP. 

\section{A Decomposition of TCUC-SDP}
\label{sec:Decomposition} 


Building upon a long tradition of decomposing MINLP into a MILP part known as the ``master'' problem 
and NLP part known as the ``sub-problem'' \cite{761909,1490608,1425598,4349079,5372005,7061968},
we present a principled decomposition in this spirit to solve TCUC-SDP.
The novelty of our approach lies in: (i) the formulation of the master as a MIQP, (ii) the formulation of the sub-problem as an SDP, (iii) the passing of information between the master problem and the sub-problem. From the master problem to the sub-problem, information are passed by ``construction'', i.e., by amending the parameters of the SDP relaxation. From the sub-problem to the master-problem, information are passed by ``cuts'', i.e., linear (scalar) inequalities to be included in the master problem.\\
The decomposition algorithm is outlined in Algorithm \ref{algo:decomp}. 

\begin{algorithm}[th]
	\caption{A Decomposition of TCUC-SDP} 
	\label{algo:decomp}
	\begin{algorithmic}[1]
		\STATE \label{Step1} Initialise master problem $\mathcal{M}$ with \eqref{M0}, as detailed in Section \ref{sec:master}.
		\STATE \label{Step2init} Initialise a branch-and-bound-and-cut procedure for $\mathcal{M}$.
		\WHILE{termination criteria not satisfied} %
		\STATE \label{Step2} Branch-and-bound-and-cut iteration: Apply branching decision, generate general-purpose cuts, solve a quadratic programming (QP) relaxation, seek a mixed-integer solution $(u, P)$ of $\mathcal{M}$ based on the QP relaxation, decide on further branching.
		\STATE \label{Step3} Lazy cut generation: Whenever the QP relaxation of $\mathcal{M}$ is feasible, check whether previously generated cuts 
			 \eqref{NRP-cut} are applicable. If so, add them to $\mathcal{M}$.
		\STATE \label{Step4} Test of feasibility: Whenever an improving integer-feasible solution $(u, P)$ is available, set up $|T|$ sub-problems $\mathcal{S}_t$ \eqref{defA}-\eqref{varBk}, where (i) the active generators are determined by $u$, (ii) power generation limits are adjusted as per ramping constraints \eqref{Rampa_Sub_1}-\eqref{Rampa_Sub}.
		Subsequently, analyse the feasibility of $\mathcal{S}_t, t \in T$:
		\begin{itemize}
			\item[a] If there exists a period $t$ for which $\mathcal{S}_t$ is infeasible, then a no-good cut \eqref{nogood} is added to $\mathcal{M}$ to exclude the schedule $u^t = \{u_k^t : k \in G\}$ from further considerations.
			\item[b)] If all $\mathcal{S}_t$ are feasible, then $(u, P)$ yields a valid upper bound on the optimum of MISDP (\ref{CoperP_noCommitmentCost}--\ref{varR} and \ref{defA}--\ref{varBk}): Save $(u, P)$. Add cuts \eqref{NRP-cut} to a pool, to be added to $\mathcal{M}$, when violated.  
		\end{itemize}
		\ENDWHILE
		\RETURN Information about the gap reached, and the best solution $(u, P)$ found so far, if there is one. 
	\end{algorithmic}
\end{algorithm}
We stress that the algorithm schema is only meant to convey the general idea. The actual implementation, as described in the following sections, is considerably more involved. A full account of the branch-and-bound-and-cut procedure of Steps~\ref{Step2init} and \ref{Step2}, in particular, is outside of the scope of this paper and we refer to \cite{wolsey1998integer,conforti2014integer} for a book-length treatment. 
Steps~\ref{Step3} and \ref{Step4} are performed asynchronously via call-backs
in the branch-and-bound-and-cut procedure.

In Figure \ref{fig-TCUC-models}, we summarize the  presented
formulations and their relationship with the master problem and sub-problems of the decomposition-based approach.
Let us now describe the master problem, the sub-problem, and the cuts in detail.

%

\subsection{The Master Problem}\label{sec:master}
In Steps \ref{Step1} and \ref{Step2} of the decomposition algorithm, one seeks improving unit commitment schedules within the model \eqref{CoperP_noCommitmentCost}-\eqref{varR}, considering the 
cost of the non-revenue power (transmission losses) computed in the sub-problem, 
but neglecting the transmission constraints \eqref{transcons} otherwise.

In particular, we introduce non-revenue power (NRP) variables $\theta_t \in \R, t=1,\dots,T$, which represent the costs of non-revenue power. Initially, they are element-wise non-negative, and included in the objective as $\sum_{t \in T} f(\theta_t)$, where $f$ is a lower envelope of quadratic functions 
for the per-generator costs of generating active power:
\begin{align}
\min \: & \eqref{CoperP_noCommitmentCost}+\eqref{CoperP_CommitmentCost}+\eqref{Carr0}+\eqref{Carr}+ \sumt f(\theta_t) 
 \tag{$\mathcal{M}_0$} \label{M0} \\
\stt& \eqref{LoadBalance}-\eqref{varR}, \theta_t \geq 0 \quad t \in T. \notag
\end{align}
The lower envelope computation is detailed in Appendix \ref{sec:low-envelopes} (on-line). 
Notice that the master problem \eqref{M0} is a mixed-integer quadratic programming (MIQP) problem.

Notice that within the branch-and-bound-and-cut, both general-purpose and TCUC-specific cuts are added to the master problem.
Whenever a lower bound $\Theta_t$ on the non-revenue power is available, the non-revenue power variables are bounded from below by non-revenue power cuts, detailed in Section \ref{sec:NRP-cuts}.
Whenever the unit commitment turns out to be infeasible, no-good cuts are added to the master problem, 
as described in Section \ref{no-good-cuts}.
We use $\mathcal{M}$ to denote the master problem augmented with such cuts.
%


\subsection{The Sub-Problem}\label{sec:sub-problem}

In Step \ref{Step4} of the decomposition algorithm, feasibility of the unit commitment schedule corresponding to the incumbent solution  $u$ for each period $t$ of the master problem is tested. To do so, per-period SDP relaxations \eqref{defA}-\eqref{varBk} (cf. Section \ref{sec:prob}) are constructed as follows:
\begin{itemize}
	\item Given the commitment schedule of $u$, the set of generators $G$ is restricted to the units that are on in $u$ at time $t$. The transmission lines are present among generators that are on and non-isolated buses. Commitment costs $c_k^0$ are then considered only for active generators $k$. 
	\item The demand at each bus $k$ is expressed with the active loads $P_k^{d,t}$ and reactive loads $Q_k^{d,t}$.
	\item For consecutive time periods, the ramping constraints \eqref{Rampa_Sub_1}-\eqref{Rampa_Sub} affect the sub-problem in terms of the limitations on active power generation. Specifically, given the value of the active power $P_k^{g,t}$ determined in the master problem solution, then $P_k^{\max, t+1}$ is set to $\min(P_k^{\max, t+1}, P_k^{g,t}+u_k^{\textrm{RampUp}})$. A similar adjustment is made for the starting period $T_i$ by referring to the initial power $u_k^{InitP}$.
\end{itemize}
Notice that the sub-problems are time-dependent restrictions of the TCUC-SDP to the integer solution $u$, as reported in \ref{fig-TCUC-models}.

Each feasible solution of the SDP sub-problems \eqref{defA}-\eqref{varBk}, for every $t \in T$, makes it possible to obtain an upper bound on the MISDP optimal value. In particular, the bound is the sum of:
\begin{itemize}
	\item the time-dependent version $O_t$ of the objective function of Optimization 4 in \cite{lavaei2012zero}, for every $t \in T$:
	 \begin{align*}
	O_t =& \sum_{k \in N} \underline{\lambda}_k^t(P_k^{\min, t}-P_k^{d, t}) - \sum_{k \in N} \overline{\lambda}_k^t (P_k^{\max, t}-P_k^{d, t} ) \\
	&+\sum_{k \in G} P^{d, t}_k (c^1_k+2B_k^{1, t}(k) \sqrt{c^2_k})\\
	&+\sum_{k \in G} \left(c^0_k -B_k^{t, 2}(k)\right)\\
	&+\sum_{k \in N} \underline{\gamma}_k^t(Q_k^{\min, t}-Q_k^{d, t} )-\sum_{k \in N}\overline{\gamma}_k^t( Q_k^{\max, t}-Q_k^{d, t})\\
	&+\sum_{k \in N} \underline{\mu}_k^t(V_k^{\min})^2-\sum_{k \in N} \overline{\mu}_k^t( V_k^{\max})^2\\
	&-\sum_{b \in \{1, \dots, 2|L|\}}  \left((S_{lm}^{\max})^2B_{lm}^{t, 1}(b)+B_{lm}^{t, 4}(b)+B_{lm}^{t, 6}(b) \right)\\
	\end{align*}
	\item unit-commitment related costs \eqref{CoperP_CommitmentCost}, \eqref{Carr}, \eqref{Carr0}.
\end{itemize}


\subsection{The Cuts}

We use cutting plans (valid inequalities, also known as cuts) to communicate information between the sub-problems and the master problem, as suggested in Figure \ref{fig-TCUC-models}. 
In particular, we present a principled approach to deriving such cuts without any assumptions of convexity on either the master problem or the sub-problem. 
Let us consider a sub-problem for a single period $t$ and the corresponding integral assignment.
We consider four kinds of cuts, depending on the outcome of the sub-problem. Specifically:
\begin{itemize}
\item If sub-problem with generators $G'$ being on and the remaining generators $G \setminus G'$ being off at time $t$ is infeasible,
a ``no-good'' cut is added, which prohibits the current subset of generators to be used in all periods with the same demand.
 \item Optionally, one can strengthen the ``no-good'' cut above, heuristically, by adding generators to $G'$ in the increasing order of their maximum power-output at $t$ until the assignment becomes feasible. 
\item If sub-problem with generators $G'$ being on at time $t$ leads to a feasible sub-problem with a lower bound $\Theta_t$ on the non-revenue power,
one can add a cut to force the value of the variable $\theta_t$ in the master problem, using $\Theta_t$ as a constant. 
 \item Additionally, a strengthening of the cut above is sought, where sub-problem partitions generators
  $G = G' \cup \bar G \cup U$,
  where generators $G'$ are on at time $t$, generators $\bar G$ are off, and generators $U$ are in $[0, 1]$. 
  Subsequently, one aims to minimise the non-revenue power $\Theta_t$ based on $G = G' \cup \bar G \cup U$ and derives a corresponding cut. 
\end{itemize}
When one assumes that there is ramping, one needs to consider the output of the generators
at $t - 1$ in the no-good cut, which complicates the presentation, but does not make the problem
substantially more difficult, in practice.

\section{Details of the Cuts}

\subsection{No-good Cuts}\label{no-good-cuts}

As usual in decomposition-based approaches, if a candidate solution of the master problem renders the sub-problem infeasible,
we prohibit its subsequent consideration in the master problem. This is achieved by adding the so-called ``no-good'' cuts (see, e.g., \cite{balas1972canonical}, \cite{d2010interval}) to the master.
In particular, if the candidate solution $S$ of the master problem has generators $G'$ on and the remaining generators $G \setminus G'$ off at time $t$,
and this turns out to be infeasible in the sub-problem, the following cut is added:
 \begin{align}
     \label{nogood}
      \sum_{g \in G'} u_{g}^t - \sum_{g \in G \setminus G'} u_{g}^t \le |G'| - 1.
 \end{align}
For each generator $g \in G'$ that is \emph{off} we accumulate $0$ on the left-hand side, while for each generator $g \in G \setminus G'$ that is \emph{on}, the left-hand side is decreased by $-1$. In this way, if all generators' statuses are the same as in $S$, then the left-hand side is $|G'|$, while it is strictly less in any other possible solution.
  We stress that this cut is valid only in periods with the same load as $t$,
 and when ramping constraints \eqref{Rampa_Sub} are not active in either the instance tested for infeasibility
 or the present relaxation.

\subsection{The No-good Cuts and Ramping Constraints}\label{no-good-cuts-ramping}

 One can derive a variant of the cut, which is valid independent of the activity of ramping constraints \eqref{Rampa_Sub}, which constrain 
 $P_k^{g,t}-P_k^{g,t-1}\leq u_k^{\textrm{RampUp}}$. 
 In order to apply the results of an infeasibility test derived with a particular $P_k^{g,t-1}$, one needs to test whether the
 current value of $P_k^{g,t-1}$ has the same impact on $P_k^{g,t}$. 
 This can be recast as testing whether $\max(0, P_k^{\max} - P_k^{g,t-1} - u_k^{\textrm{RampUp}})$ is the same using the 
 current value of $P_k^{g,t-1}$ and the value used in the infeasibility test.
 (Note that in case $P_k^{\max} \le P_k^{g,t-1} + u_k^{\textrm{RampUp}}$, the ramping constraints are not active.)
 
 For each distinct value $V_k^{t,d} := \max(0, P_k^{\max} - P_k^{g,t-1} - u_k^{\textrm{RampUp}})$ encountered at an incumbent solution with value $P_k^{g,t-1}$, 
 we introduce a new binary variable, where the $d$th variable for a given time-generator pair $(t, g)$, we denote $v_k^{t,d} \in [0, 1]$.
 The intended meaning is that we wish to have $v_k^{t,d}$ of 1 if an only if we counter $V_k^{t,d}$ in the relaxation.
 This can be cast as: 
 \begin{align}
     \label{VbigM1}
  P_k^{\max} - P_k^{g,t-1} - u_k^{\textrm{RampUp}} - V_k^{t,d} \le M - M v_k^{t,d} + \epsilon \\ 
 - P_k^{\max} + P_k^{g,t-1} + u_k^{\textrm{RampUp}} + V_k^{t,d} \le M - M v_k^{t,d} + \epsilon 
     \label{VbigM2}
 \end{align}
 for a sufficiently large $M \in \R$ and a sufficiently small $\epsilon \in \R$.
 This could be strengthened further using the perspective reformulation \cite{Frangioni2006}.
 Additionally, we set $0 \le \sum_d v_k^{t,d} \le 1$.
 
 Using the variable $v_k^{t,d}$, we can make the no-good cut \eqref{nogood} specific to a particular setting of $V_k^{t,d}, \forall k, d$ as:
  \begin{align}
     \label{nogood-ramping}
      \sum_{g \in G'} u_{g}^t + \sum_{g \in G} v_g^{t,d} - \sum_{g \in G} \sum_{e \not = d} v_k^{t,e} - \sum_{g \in G \setminus G'} u_{g}^t \le 2 |G'| - 1
 \end{align}
 which marks the unit commitment vector as infeasible only if $v_g^{t,d}$ is 1 at all generators $g$ for the correct $d$,
 and hence only if $P_k^{g,t-1}$ has the same value as used in the infeasibility test for all generators $g$.
 
 

\subsection{Non-
	Revenue Power (NRP) Cuts}
\label{sec:NRP-cuts}

In order to link the master problem and the sub-problem, we consider cuts in the master problem 
that force a newly introduced variable to a certain value in case of a certain unit commitment. 
To derive such a cut, notice that in \eqref{nogood}, we have an expression that takes the value 1 if and only if the candidate solution is
 generators $G'$ on and the remainder $G \setminus G'$ off:

\begin{align}
    \label{nogood-expr}
    1 - |G'| + \sum_{g \in G'} u_{g}^t - \sum_{g \in G \setminus G'} u_{g}^t = \begin{cases}
    1 & \textrm{ for candidate } S \\
    \le 0 & \textrm{ otherwise}
    \end{cases}
\end{align}

We can use the expression \eqref{nogood-expr} in an inequality forcing, in the candidate solution with a certain unit commitment, a newly-added non-negative variable $\theta_t$ to be greater than a lower bound $\Theta_t$ specific to that unit commitment:

\begin{align}
    \label{NRP-cut}
    \Theta_t (1 - |G'| + \sum_{g \in G'} u_{g}^t - \sum_{g \in G \setminus G'} u_{g}^t) \le \theta_t
\end{align}

Observe that if the expression \eqref{nogood-expr} evaluates to 0 or a negative number, the value of $\theta_t$ is not changed,
as it is bounded from below by 0 in any case.
If, however, the expression \eqref{nogood-expr} evaluates to 1, the cut \eqref{NRP-cut} forces $\theta_t$ to be at least
$\Theta_t$, i.e., the value obtained in the sub-problem.
Notice that this cut can be applied only in the same period $t$, or other periods with the same power generation as in $t$.

One option is to consider term 
\begin{align} \sumt \max \left \{ \theta_t, \sumk (c_k^2 (P^{g,t}_k)^2 + c_k^1 P^{g,t}_k) \right \} \end{align}
in the master problem, where $\theta_t$ is at least the objective $O_t$ of the SDP, for all $t \in T$. 
Instead, we consider a lower bound on the non-revenue power (NRP) and its transfer from the sub-problem to the master problem in 
non-revenue power (NRP) variables $\theta_t \in \R, t=1,\dots,T$
using cuts we refer to as ``non-revenue power cuts''.
There, we need to consider a lower envelope
 $f$ of all the quadratic objective
functions across the generators $G'$ and use $f(\theta_t)$ in the objective. 
The computation of the lower envelope amounts to a small linear program,
which we detail only in Appendix~\ref{sec:envelopes} (on-line), together with the computation of the lower bound $\Theta_t$. 
Further in Appendix~\ref{no-good-cuts-ramping} (on-line), we show how to derive a variant of the cut, which is valid independent of the activity of ramping constraints \eqref{Rampa_Sub}. 


\section{An Analysis of Convergence}
\label{sec:Convergence} 

The analysis is straightforward, if one can make two assumptions, the first of which is:

\begin{assumption}
\label{ass1}
There exists a test of feasibility of a semidefinite program (\ref{defA}--\ref{varBk}), 
computable with no error in finite time.
\end{assumption}

In theory, it is known \cite{Ramana1997} that the test is in 
NP and Co-NP simultaneously, in both the Turing machine model
as well as in the real-number model of Blum, Shub and Smale.
(Specifically, see Theorem 3.4 in  \cite{Ramana1997}. The certificates of infeasibility are polynomially short 
and the feasibility test is not NP-Complete unless NP = Co-NP.)
For practical purposes, this is also not too bad an assumption, 
considering feasibility tests correct up to the machine precision exist.
In our implementation, we consider only the machine precision.

\begin{assumption}
\label{ass2}
All integer-feasible points of the master problem (\ref{CoperP_noCommitmentCost}-\ref{varR})
can be enumerated in finite time.
\end{assumption}

In theory, this assumption is satisfied, whenever for
each assignment of the binary variables, there is a unique value of
the remaining solutions
and we can evaluate the corresponding objective function in finite time. 
This is the case for the master problem (\ref{CoperP_noCommitmentCost}-\ref{varR}), indeed.
In pratice, the IBM ILOG CPLEX routine {\tt populate}
\footnote{\url{https://www.ibm.com/support/knowledgecenter/en/SSSA5P_12.6.3/ilog.odms.cplex.help/CPLEX/UsrMan/topics/discr_optim/soln_pool/18_howTo.html}},
which we use in our implementation, is capable of the enumeration.
Then:

\begin{theorem}
Under Assumptions \ref{ass1} and \ref{ass2},
for any instance of MISDP (\ref{CoperP_noCommitmentCost}--\ref{varR} and \ref{defA}--\ref{varBk}), 
there exists a finite integer $i$, such that after $i$ calls to the test of feasibility of 
a SDP (\ref{defA}--\ref{varBk}), 
one is guaranteed to have obtained a global optimum of (\ref{CoperP_noCommitmentCost}--\ref{varR} and \ref{defA}--\ref{varBk})
or a certificate that no feasible solution exists.
\end{theorem}

\begin{proof}
Notice that there exists a finite number of integer points in the feasible region of the master problem (\ref{CoperP_noCommitmentCost}-\ref{varR}).
In the worst case, i.e., with both the quadratic relaxation of the master problem and the SDP relaxation in the sub-problem 
having the same objective function value at all nodes of the search tree, 
we would have to visit each node of the search tree, and hence each integer point in the feasible region of the master problem (\ref{CoperP_noCommitmentCost}-\ref{varR})
precisely once. 
Once we do so, we can pick the optimum.
Considering the number of nodes to visit is finite, 
the number of calls to the feasibility test is also finite,
and the overall run-time is also finite.
\end{proof}

Without the Assumption \ref{ass1}, which could be seen as too restrictive, 
we could consider:

\begin{assumption}
\label{ass3}
Each SDP (\ref{defA}--\ref{varBk}) considered in the test of feasibility are either infeasible or 
there exists a feasible ball of radius at least exp(-poly($n,m,L$)) where $n$ 
 is the dimension, 
$m$ is the number of constraints, and $L$ is the input size, inside the relaxation.
\end{assumption}

which could be seen as a certain measure of robustness of the solution of the SDP.
Then, again:

\begin{theorem}
Under Assumptions \ref{ass2} and \ref{ass3},
for any instance of MISDP (\ref{CoperP_noCommitmentCost}--\ref{varR} and \ref{defA}--\ref{varBk}), there exists a finite integer $j$, 
such that after $j$ calls to the test of feasibility of an SDP (\ref{defA}--\ref{varBk}),
one is guaranteed to have obtained a global optimum of (\ref{CoperP_noCommitmentCost}--\ref{varR} and \ref{defA}--\ref{varBk})
or a certificate that no feasible solution exists.
\end{theorem}

Further, one could study the conditions, which have to be satisfied in order for the feasible region of (\ref{CoperP_noCommitmentCost}--\ref{varR} and \ref{defA}--\ref{varBk}) 
to coincide with the feasible region of (\ref{CoperP_noCommitmentCost}--\ref{varR} and \ref{defPgTrans}--\ref{defPgTransLast}). Clearly:

\begin{proposition}
Whenever the feasible region of the SDP relaxation (\ref{defA}--\ref{varBk}) 
at each node of the search tree coincides with the POP (\ref{defPgTrans}--\ref{defPgTransLast}) at that node of the search tree,
the feasible region of the MISDP (\ref{CoperP_noCommitmentCost}--\ref{varR} and \ref{defA}--\ref{varBk})
coincides with the feasible region of the MIPOP (\ref{CoperP_noCommitmentCost}--\ref{varR} and \ref{defPgTrans}--\ref{defPgTransLast}).
\end{proposition}

This is the case, for example, if the power-system is radial and homogeneous (cf. Theorem 3 in \cite{6624135} and Theorem 1 in \cite{dymarsky2018convexity}), 
or if there are enough phase-shifters (cf. Theorem 4 of Sojoudi and Lavaei \cite{6345272}).
See \cite{6815671} for further references. 
Clearly, these conditions are invariant throughout the nodes of the search tree, 
and hence if they are satisfied by the SDP at one node of the branch-and-bound 
search tree, they are satisfied at all nodes. 
Alternatively, one could assume that whenever the value of the POP (\ref{defPgTrans}--\ref{defPgTransLast})
for one input is higher than at another, then the value of the SDP (\ref{defA}--\ref{varBk}) 
at that input is higher than at the other. 
This assumption seems much weaker, albeit perhaps more difficult to work with. 

\section{Computational Illustrations}
\label{sec:Implementation}

We have conducted a numerical evaluation of the decomposition on instances
of Red Electrica de Espa{\~n}a. 
The input to the master problem comes in the
form of a MILP model for the unit commitment problem without transmission constraints. 
The input
to the sub-problem comes in the Common Information
Model (CIM, IEC 61970) serialisation in Extensible Markup
Language (XML), which is used by the European Network
of Transmission System Operators (ENTSO-E) to exchange
operational data.
The branch-and-bound-and-cut is based on IBM
ILOG CPLEX 12.7. Whenever IBM ILOG CPLEX obtains a
new QP relaxation, a call-back checks whether any cuts are
applicable, using the previously stored information about the
outcomes of the sub-problems. Notice that the cuts \eqref{NRP-cut} are
derived for each period independently, and hence cannot be
particularly dense, but the addition of a excessive amount of
cuts would still affect the performance negatively. Whenever
IBM ILOG CPLEX encounters an improving integer-feasible
solution while solving the master problem, another call-back
constructs the $|T|$  sub-problems, solves them, and stores
information about the outcomes. To solve these sub-problems, we used a
custom interior-point method (IPM) for solving the resulting
SDP relaxations.

For illustration purpose, we consider an instance capturing a portion of the Canary Islands transmission system, as illustrated in Figure \ref{canary-network}. 
The instance, which we denote \textit{REE}, features $45$ buses, of which $31$ are generators, 
and $48$ branches, of which $4$ are double circuits. 
There are $|T| = 24$ time-periods, corresponding to 24 hours in a representative day 
of operations in January 2017, when the loads have been recorded. 

	\begin{figure}[h!]
	\begin{center} 
		\includegraphics[width = 0.7\textwidth, height = 5cm, keepaspectratio]{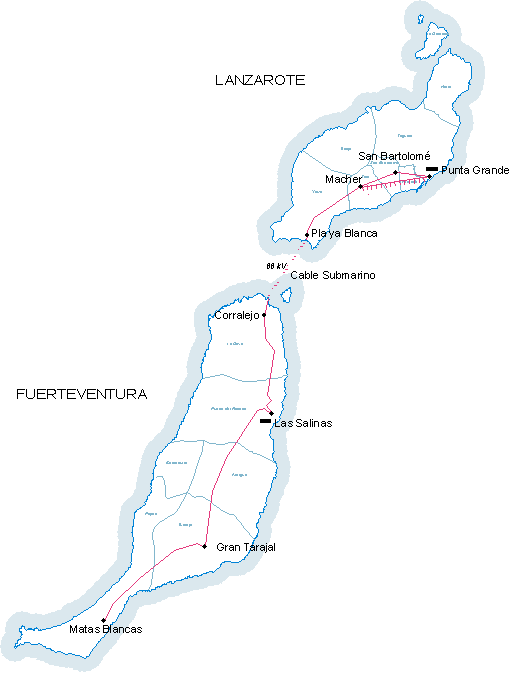}
		\caption{A portion of the Canary Islands transmission system.}
		\label{canary-network}
	\end{center}
\end{figure}
In evaluating the performance of our approach on an instance, one may consider:
\begin{itemize}
	\item Lower and upper bounds on the optimum of the master problem (\ref{CoperP_noCommitmentCost}-\ref{varR}). Since the master problem is a relaxation of the MISDP (\ref{CoperP_noCommitmentCost}--\ref{varR} and \ref{defA}--\ref{varBk}), any lower bound for the master problem is a valid lower bound for the MISDP.
	\item The so-called \enquote{sub-problem bound} sums objective function values of per-period SDPs (\ref{defA}--\ref{varBk}) across all time periods, alongside the commitment-related sunk costs \eqref{CoperP_CommitmentCost}, \eqref{Carr0}, \eqref{Carr}. 
	 The sub-problem bound is an upper bound on the MISDP (\ref{CoperP_noCommitmentCost}--\ref{varR} and \ref{defA}--\ref{varBk}) optimal value.
\end{itemize}


A sample evolution of bounds obtained for the \textit{REE} instance by running the implementation 
on a standard laptop (Lenovo ThinkPad X1 Carbon) without any explicit paralellisation
 is presented in Figure~\ref{fig:convergence}. 
Notice that the first schedule feasible with respect to transmission constraints (TC) is found within $600$ seconds of computation, and $7$ solutions of progressively better quality are found within $4800$ seconds. The TC-feasibility has been confirmed by an independent run of a power-flow software utilised by Red El{\' e}ctrica de Espa\~na. 
It should be noted that the bulk of the run-time is down to solving $192$ per-period SDPs (\ref{defA}-\ref{varBk}), which takes $20-30$ seconds per SDP. Since the SDP sub-problems are formulated based on the master problem solution, they could be solved in parallel, which would speed-up the implementation considerably.
In Figure~\ref{fig:master}, the master problem convergence bounds are reported for the same run of the decomposition approach. While the lower bound in the master problem does not improve particularly fast, a gap of $3.83\%$ in the master problem is reached within $1677$ seconds. 


\begin{figure}[h!]
\begin{center}
	\includegraphics[width = 0.5\textwidth]{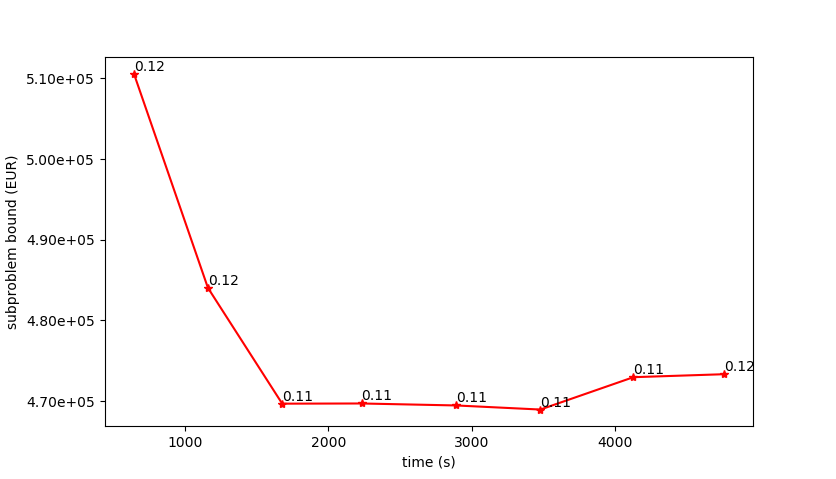}
\end{center}
	\caption{The performance of the decomposition approach on the \emph{REE} instance. On the vertical axes, we plot values obtained in the sub-problem. The additional annotations suggest the primal-dual gaps. 
}
	\label{fig:convergence}
\end{figure}

\begin{figure}[h!]
	\begin{center} 
		\includegraphics[width = 0.5\textwidth]{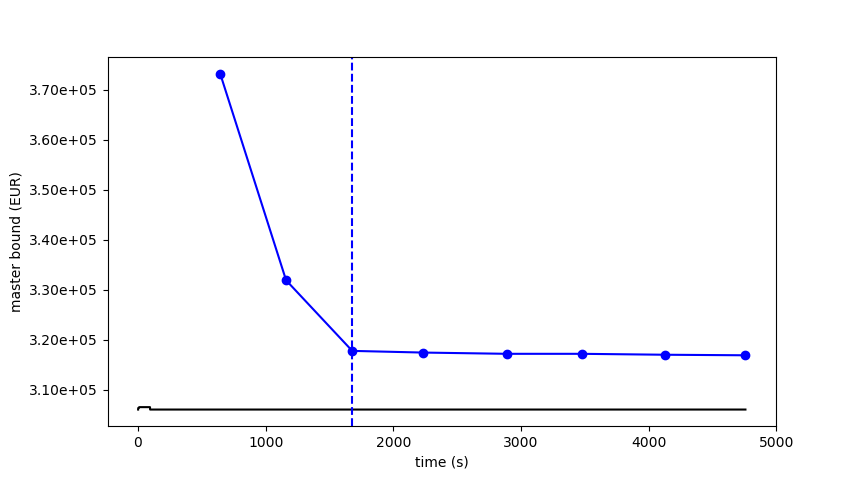}
	\end{center}
	\caption{Lower and upper bounds on the master problem found using the decomposition approach.}
\label{fig:master}
\end{figure}

\section{Conclusions and Future Work}

The relevance of alternating-current transmission constraints in unit commitment
increases with the introduction of intermittent renewable energy sources and decreasing market
depth in traditional energy sources.  
Based on a mixed-integer SDP formulation, we have presented a decomposition approach,
which interleaves solving a master problem and sub-problems, 
passing information between them in the form of no-good cuts and non-revenue power cuts, in one way, 
and limits on generation in the sub-problems, the other way.
For the first time, we have analysed convergence of such an approach with non-linear sub-problems. 
On an instance from the Canary Islands, our implementation runs  
within realistic time-frames even on a laptop.

We hope that this may spur a considerable amount of further research.
First, the master problem could be tightened.
While many of the techniques developed for the unit commitment without transmission constraints \cite{4682641,5983423,6407169,7798349,8275570}
are applicable, one could also consider the Karush--Kuhn--Tucker (KKT) conditions of the sub-problem \cite{Burer2008} in the master problem,
 or even constraints valid only at the rank-1 solutions \cite{Zheng2011} of the sub-problem. 
One may also aim to control the number of commitment schedules considered, e.g.,
 by adjusting the branch-and-bound-and-cut procedure\footnote{In IBM ILOG CPLEX, the relevant parameters include {\tt CPX\_MIPEMPHASIS\_FEASIBILITY}.}, 
 or considering only some of the candidate solutions encountered in the master problem.
Finally, in order to speed up the implementation, 
 further work could focus on parallel and distributed solvers for the sub-problem,
 beyond solving the $|T|$ SDPs in parallel, when no ramping constraints are active. 
Eventually, this could make clearing continental-scale markets under transmission constraints possible.


\section*{Acknowledgment}
\addcontentsline{toc}{section}{Acknowledgment}

We would like to thank 
Jos{\' e} Javier P{\' e}rez Gonz{\' a}lez,
Marta Quintana de Juan,  
Irene Paramio Lorente, 
Laura Ganzabal Fern{\' a}ndez, 
and Francisco Jos{\' e} Riola Hernang{\' o}mez
for their support during this collaboration.

\bibliographystyle{ieeetran}
\bibliography{../tcuc,../enclosing}

\clearpage
\appendix
\section{Additional Material}

\subsection{Strengthening of No-good Cuts}\label{no-good-cuts-strength} 

A strengthening of the ``no-good'' cut above is sought, heuristically, by moving generators from $G \setminus G'$ to $G'$ in the 
increasing order of the maximum power-output of $g \in G \setminus G'$ at $t$, until (and not including) the $G'$ 
makes the sub-problem becomes feasible.
Although it seems hard to provide theoretical guarantees as to the performance, this seems to be surprisingly successful.

A further strengthening 
in the case of ramping constraints 
can be obtained, when one moves from linear cuts to second-order cuts using the so called
perspective reformulation \cite{Frangioni2006}. There, good performance guarantees can be obtained.

\subsection{The Non-Revenue Power Cuts and Ramping Constraints}

As in Section~\ref{no-good-cuts-ramping}, one can derive a variant of the cut, which is valid independent of the activity of ramping constraints \eqref{Rampa_Sub}. 
Using the same value $V_k^{t,d} := \max(0, P_k^{\max} - P_k^{g,t-1} - u_k^{\textrm{RampUp}})$, variable $v_k^{t,d} \in [0, 1]$,
and constraints (\ref{VbigM1}--\ref{VbigM2}), we can extend the NRP cut \eqref{NRP-cut}
specific to a particular setting of $V_k^{t,d}, \forall k, d$ as:
  \begin{align}
      \Theta_t (1 - 2 |G'| + \sum_{g \in G'} u_{g}^t + \sum_{g \in G} v_g^{t,d} - \notag\\
      - \sum_{g \in G \setminus G'} u_{g}^t - \sum_{g \in G} \sum_{e \not = d} v_k^{t,e} ) \le \theta_t           \label{NRP-ramping}
 \end{align}
 which pushes up the value of $\theta_t$ only if $v_g^{t,d}$ is 1 at all generators $g$ for the correct $d$,
 and hence only if $P_k^{g,t-1}$ has the same value as used in the infeasibility test for all generators $g$.

\subsection{Strengthening of Non-Revenue Power Cuts}\label{NRP-cuts-strength} 

 Just as in the case of no-good cuts, one may try to add generators to $G'$, in the NRP cuts, one
 may try removing them from $G'$. If we amend the sub-problem to accept a partition of the generators
  $G = G' \cup \bar G \cup U$,
  where generators $G'$ are on at time $t$, generators $\bar G$ are off, and generators $U$ are in $[0, 1]$,
  one obtains a lower bound $\Theta_t$ on the non-revenue power at time $t$.
 With NRP $_t$ based on $G = G' \cup \bar G \cup U$, one can apply the cut:
   \begin{align}
\label{NRP-cut2}
   \Theta_t (1 - |G'| + \sum_{g \in G'} (u_{g}^t + \sum_{g \in \bar G} (1-u_{g}^t)) \le \theta_t.
   \end{align}
  Note that one may derive the set $U$, heuristically, by considering the  commitment of unit and time-period pairs, where the minimum and maximum
 on and off time constraints are not active, for instance. 
Again, a further strengthening 
can be obtained using the 
perspective reformulation \cite{Frangioni2006}. 

\subsection{Per-generator Non-Revenue Power Cuts}

One could also use the expression \eqref{nogood-expr} to bound per-generator non-revenue power from below, in theory.
Let us have a variable $l_{g,t}$ for NRP at $g \in G'$, period $t$ in the master problem,
and a lower bound thereupon obtained from the sub-problem. Then:

\begin{align}
    \label{NRP-cut3}
    \Theta_{g,t} (1 - |G'| + \sum_{g \in G'} u_{g}^t - \sum_{g \in G \setminus G'} u_{g}^t) \le \theta_{g,t}
\end{align}

extends \eqref{NRP-cut} to this setting.
That is: if the expression \eqref{nogood-expr} evaluates to 0 or a negative number, the value of $\theta_{g,t}$ is not changed,
as it is bounded from below by 0 in any case.
If, however, the expression \eqref{nogood-expr} evaluates to one, the cut \eqref{NRP-cut3} forces $\theta_{g,t}$ in the master problem
to be at least $\Theta_{g,t}$, i.e., the value obtained in the sub-problem.
Notice that this cut can be applied only in the same period $t$, or other periods with the same power generation as in $t$.\\
We stress that at the moment, we have no proof that the 
per-generator NRP cuts,
except the intuition that the choices of powers in the relaxation \emph{together} form a lower bound.
By default, the per-generator NRP cuts are not enabled in our implementation.

\section{Lower and Upper Envelopes of the Quadratic Costs}
\label{sec:envelopes}

Notice that the approach talks about the lower bounds on the non-revenue power incurred or
on the active power generation at individual generators, neither of which is generally
easy to obtain in non-linear programming.
While the objective of the SDP relaxation is a valid lower bound
on the objective of the quadratically constrained quadratic program 
\cite{Fujie1997,Nesterov2000}, 
there is no way of lower bounding individual variables.
 
Crucially, we observe, however, that the objective function of TCUC-SDP is a sum of quadratic functions, whose range are non-negative real numbers.
In the following subsections, we show how to approximate both the lower and upper envelope of the quadratic functions to arbitrary precision by running a small linear programming problem.
We can then use: (i) the upper envelope to compute a lower bound on the NRP and (ii) the lower envelope to include an estimation of the NRP in the objective of the master problem.

\subsection{Lower Envelopes}
\label{sec:low-envelopes}

As has been explained in Section~\ref{sec:NRP-cuts},
in the cost function of the master problem, we include a lower bound on the cost of the non-revenue power.
In the simplest case, this can take the form of 
 $\sum_{t = 1}^T f(l_t)$, where $f$ is a lower envelope of quadratic functions 
for the per-generator costs of generating active power and $l_t$ is a variable bound from 
below by a lower bound on the
 amount of active power lost in transmission at time $t$.
 
We compute the tightest quadratic lower envelope $f$ of the cost functions
$C_g (x) := c_g^2 x^2+c_g^1 x+c_g^0$ in terms of active power $x \in \R$ at all generators $g \in G$, 
 which has the form $f(x)=c x^2 + b x + a$, $x \in \R$.
One can compute an arbitrarily good approximation of a lower envelope by considering 
 $f$ such that $f(x) \leq C_g(x), \forall x \in X, g \in G$, where 
 $X$ is a discretization $X=\{x_1, \dots, x_p\}$ of the set of possible values of active power. 
Out of all such lower envelopes $f(x)$, we are interested the tightest lower envelope, 
hence we maximize the values of $f(x)$ summed across all points in $X$. This amounts
to solving the following linear problem in dimension 3 for the 3 coefficients $a,b,c$:
\begin{align}
\displaystyle\max_{a,b,c \in \R}\ & \displaystyle\sum_{x \in P} \left( c x^2 + b x + a \right)   &\notag\\
\stt & c x^2  + b x + a \leq \displaystyle\min_{g \in G} (c_g^2 x^2+c_g^1 x + c_g^0)  & \forall x \in X \notag\\
& a,b,c \geq 0 & \notag
\end{align}

%

\subsection{Upper Envelopes}
\label{sec:up-envelopes}
The lower bound on the total cost in the SDP relaxation obtained by means of the lower envelopes may be overly conservative in cases where one $g$ generator is much more expensive than others, but the upper limit $P_g^{\max}$ on the generation of active power at $g$ is a small fraction of the total demand.\\
A tighter lower bound on the non-revenue power could be obtained using an
upper envelope of the quadratic functions 
for the per-generator costs of generating active power. 
Specifically, this is achieved by considering the generators in the decreasing order of their costs for the maximum 
active power generation, decrementing a bound on the costs (e.g., the objective of the sub-problem) by the cost of the generation
at the maximum active power generation, and incrementing the accounted-for active power, until the accounted-for active power matches 
the total power demand.
The remainder of the (bound on the) costs are associated with the NRP.  
To compute the amount of NRP, one can once again work with the remaining generators, 
in the decreasing order of their costs for the maximum 
active power generation.
Algorithm \ref{lowerbound} presents the pseudo-code.

\begin{algorithm}[H]
\caption{Computation of a Lower Bound on the Non-Revenue Power} 
\label{lowerbound}
\begin{algorithmic}[1]
\REQUIRE Objective $O_t$ of the SDP relaxation at time $t$, aggregate power demand $P^{d,t}$ at time $t$,
 cost functions $C_g^t$, limits $P_g^{\max, t}$
\STATE{Sort generators $g \in G$ in the decreasing order of $C_g(P_g^{\max})/P_g^{\max}$}
\STATE{Initialise $g$ to be the first generator in this order, costs $O \from O_t$, active power $\widetilde{P} \from P^{d,t}$, a lower bound on the NRP $L \from 0$}
\WHILE{ $O > 0$ and $\widetilde{P} > 0$ }
\STATE $\widetilde{P} \from \widetilde{P} - \min(\widetilde{P}, P_g^{\max})$
\STATE $O \from O -  \min( C_g(P_g^{\max}), C_g(\widetilde{P}) )$
\STATE $g \from$ succ$(g)$ in the ordered $G$, if defined
\ENDWHILE
\WHILE{ $O > 0$ }
\STATE $L \from L + \min(P_g^{\max}, \arg \max_x C_g(x)$ s.t. $C_g(x) = O)$
\STATE $O \from O -  \min( C_g(P_g^{\max}), O )$
\STATE $g \from$ succ$(g)$ in the ordered $G$, if defined
\ENDWHILE
\RETURN L
\end{algorithmic}
\end{algorithm}

Notice that in the case where the cost functions are similar, both the upper envelope and this procedure should give a very good bound. 
This procedure produces a useful bound, even if the 
cost functions vary, as long as there are reasonably tight bounds on the active-power output of each generator.

Due to its construction, the objective function of the SDP relaxation represents the costs related to active power generation. Therefore, an upper bound for the cost function of the Master problem is considered using feasible SDP instances.

\end{document}